\documentclass{article}

\usepackage{latexsym,amstext,amsmath,amssymb,amsthm,bbm}

\title{Multidegree for bifiltered $D$-modules and hypergeometric systems}

\author{R\'emi Arcadias$^*$\\
Depto.\ de \'Algebra, Fac.\ Matem\'aticas, Campus Reina Mercedes\\
41012 Sevilla (Spain)\\
* E-mail: rarcadias@us.es}

\date{}

\newtheorem{proposition}{Proposition}[section]
\newtheorem{theorem}{Theorem}[section]
\newtheorem{lemma}{Lemma}[section]

\theoremstyle{definition}\newtheorem{example}{Example}[section]
\theoremstyle{definition}\newtheorem{definition}{Definition}[section]
\theoremstyle{definition}\newtheorem{question}{Question}[section]

\begin{document}

\maketitle

\begin{abstract}
In that paper, we recall the notion of the multidegree for $D$-modules, as exposed in a previous paper\cite{moi11}, with a slight simplification. 
A particular emphasis is given on hypergeometric systems, used to provide interesting and computable examples. 
\end{abstract}


\section*{Introduction}

This paper is an introduction to the theory of the multidegree for $D$-modules, as exposed in a previous paper\cite{moi11}. The multidegree has been defined by E.\ Miller\cite{stu05}: that is a generalization in multigraded algebra of the usual degree known in projective geometry. In our previous paper\cite{moi11} we adapted it to the setting of bifiltered modules over the ring $D$ of linear partial differential operators. Here the definition of the multidegree is slightly simplified: it becomes the identical counterpart, in the category of bifiltered $D$-modules, of the definition of Miller. We give detailed examples from the theory of $A$-hypergeometric systems of I.M.\ Gelfand, A.V.\ Zelevinsky and M.M.\ Kapranov\cite{GKZ}.

The paper is organized as follows. In Section 1 we recall the definition of the objects we are interested in, up to the definition of $(F,V)$-bifiltered free resolutions of $D$-modules, following T.\ Oaku and N.\ Takayama\cite{oaku01}. Then we define the multidegree for bifiltered $D$-modules in Section 2. Some examples of $A$-hypergeometric systems are discussed in Section 3, leading to open questions which generalize known facts about the holonomic rank of $A$-hypergeometric systems. In Section 4 we give details on the simplification of the definition of the multidegree we give in Section 2, which consists in proving that a $D$-module $M$ and its homogenization $\mathcal{R}_V(M)$ with respect to the $V$-filtration have same codimension.

\section{Bifiltered free resolution of $D$-modules}

Let $D=\mathbb{C}[x_1,\dots,x_n,t_1,\dots,t_p]\langle\partial_{x_1},\dots,\partial_{x_n},\partial_{t_1},\dots,\partial_{t_p}\rangle$ denote the Weyl algebra in $n+p$ variables. Denoting the monomial 
\[
x_1^{\alpha_1}\dots x_n^{\alpha_n}
t_1^{\mu_1}\dots t_p^{\mu_p}
\partial_{x_1}^{\beta_1}\dots\partial_{x_n}^{\beta_n}
\partial_{t_1}^{\nu_1}\dots\partial_{t_p}^{\nu_p}
\]
by $x^{\alpha}t^{\mu}\partial_x^{\beta}\partial_t^{\nu}$, every element $P$ in $D$ is written uniquely as a finite sum with complex coefficients
\begin{equation}\label{eq1}
P=\sum a_{\alpha,\beta,\mu,\nu}
x^{\alpha}t^{\mu}\partial_x^{\beta}\partial_t^{\nu}.
\end{equation}
The fundamental relations are $\partial_{x_i}x_i=x_i\partial_{x_i}+1$ for $i=1,\dots,n$ and 
$\partial_{t_i}t_i=t_i\partial_{t_i}+1$ for $i=1,\dots,p$.

We describe two importants filtrations on $D$. For $\lambda=(\lambda_1,\dots,\lambda_r)\in\mathbb{N}^r$, let $|\lambda|=\lambda_1+\dots+\lambda_r$.
Let $\mathrm{ord}^F(P)$ (resp.\ $\mathrm{ord}^V(P)$) be the maximum of 
$|\beta|+|\nu|$ (resp.\ $|\nu|-|\mu|$) over the monomials $x^{\alpha}t^{\mu}\partial_x^{\beta}\partial_t^{\nu}$ appearing in 
(\ref{eq1}). Then we define the filtrations 
\[
F_d(D)=\{P\in D, \mathrm{ord}^F(P)\leq d\}
\]
 for $d\in\mathbb{Z}$ (with $F_d(D)=\{0\}$ for $d<0$) and 
\[
V_k(D)=\{P\in D, \mathrm{ord}^V(P)\leq k\}
\]
for $k\in\mathbb{Z}$. The filtration $(F_d(D))$ is the most classical one, and the filtration $(V_k(D))$ is the so-called $V$-filtration along $t_1=\dots=t_p=0$ of Kashiwara-Malgrange.

The graded ring 
$\mathrm{gr}^F(D)=\oplus_d F_d(D)/F_{d-1}(D)$ is a commutative polynomial ring, whereas the graded ring $\mathrm{gr}^V(D)=\oplus_k V_k(D)/V_{k-1}(D)$ is isomorphic to $D$. 

Let $M$ be a left finitely generated $D$-module. We will define the notion of a good $F$-filtration of $M$. For $\mathbf{n}=(n_1,\dots,n_r)\in\mathbb{Z}^r$, let $D^r[\mathbf{n}]$ denote the free module $D^r$ endowed with the filtration 
\[
F_d(D^r[\mathbf{n}])=\oplus_{i=1}^r F_{d-i}(D). 
\]
A \emph{good} $F$-filtration of $M$ is a sequence $(F_d(M))_{d\in\mathbb{Z}}$ of sub-vector spaces of $M$ such that 
there exists a presentation 
\[
M\stackrel{\phi}{\simeq}\frac{D^r}{N},
\]
 with $N$ a sub-$D$-module of $D^r$, and a vector shift $\mathbf{n}\in\mathbb{N}^r$, such that
 \[
F_d(M)\stackrel{\phi}{\simeq} \frac{F_d(D^r[\mathbf{n}])+N}{N}.
\] 
The module $\mathrm{gr}^F(M)=\oplus_d F_d(M)/F_{d-1}(M)$ is a graded finitely generated module over $\mathrm{gr}^F(D)$. It is proved that the radical of the annihilator of $\mathrm{gr}^F(M)$ does not depend on the good filtration of $M$. Then  
$\mathrm{codim} M$ is defined as the codimension of the ring $\mathrm{gr}^F(D)/\mathrm{Ann}(\mathrm{gr}^F(M))$, which does not depend on the good filtration. A fundamental fact is that $\mathrm{codim}M\leq n+p$ if $M\neq 0$. When $\mathrm{codim}M=n+p$, the module $M$ is said to be holonomic.

Now we introduce the bifiltration 
\[
F_{d,k}(D)=F_d(D)\cap V_k(D)
\]
 for $d,k\in\mathbb{Z}$, and we define the notion of a good $(F,V)$-bifiltration of $M$. For $\mathbf{n}=(n_1,\dots,n_r)$ and $\mathbf{m}=(m_1,\dots,m_r)\in\mathbb{Z}^r$, let $D^r[\mathbf{n}][\mathbf{m}]$ denote the free $D$-module $D^r$ endowed with the bifiltration 
\[
F_{d,k}(D^r)=\oplus_{i=1}^r F_{d-i,k-i}(D). 
\]
A \emph{good} $(F,V)$-bifiltration of $M$ is a sequence $(F_{d,k}(M))_{d,k\in\mathbb{Z}}$ of sub-vector spaces of $M$ such that 
there exists a presentation 
\[
M\stackrel{\phi}{\simeq}\frac{D^r}{N},
\]
 with $N$ a sub-$D$-module of $D^r$, and vector shifts $\mathbf{n},\mathbf{m}\in\mathbb{N}^r$, such that
 \[
F_{d,k}(M)\stackrel{\phi}{\simeq} 
\frac{F_{d,k}(D^r[\mathbf{n}][\mathbf{m}])+N}{N}.
\] 
We are now in position to define the notion of a bifiltered free resolution of a module $M$ endowed with a good bifiltration $(F_{d,k}(M))_{d,k}$. Such a resolution is an exact sequence
\[
0\to D^{r_{\delta}}[\mathbf{n}^{(\delta)}][\mathbf{m}^{(\delta)}]
\stackrel{\phi_{\delta}}{\to}\cdots
\stackrel{\phi_1}{\to} D^{r_{0}}[\mathbf{n}^{(0)}][\mathbf{m}^{(0)}]
\stackrel{\phi_0}{\to} M\to 0
\]
such that for any $d,k\in\mathbb{Z}$, we have an exact sequence
\[
0\to F_{d,k}(D^{r_{\delta}}[\mathbf{n}^{(\delta)}][\mathbf{m}^{(\delta)}])
\stackrel{\phi_{\delta}}{\to}\cdots
\stackrel{\phi_1}{\to} F_{d,k}(D^{r_{0}}[\mathbf{n}^{(0)}][\mathbf{m}^{(0)}])
\stackrel{\phi_0}{\to} F_{d,k}(M)\to 0.
\]

\begin{example}\label{ex1}
Let $D=\mathbb{C}[t_1,t_2]\langle \partial_{t_1},\partial_{t_2}\rangle$.
Let $I$ be the ideal generated by $\partial_{t_1}-\partial_{t_2}$ and $t_1\partial_{t_1}+t_2\partial_{t_2}$, and $M=D/I$ endowed with the good bifiltration
$F_{d,k}(M)=(F_{d,k}(D)+I)/I$. That is the hypergeometric system $M_A(0,0)$ associated with 
$A=(\begin{array}{cc}
1 & 1
\end{array} ),$ 
see Section 3. 
We have a bifiltered free resolution:
\[
0\to D^1[2][1]\stackrel{\phi_2}{\to} D^2[1,1][1,0]\stackrel{\phi_1}{\to} D^1[0][0]\to M\to 0
\]
with 
\begin{eqnarray*}
\phi_1(1,0) & = & \partial_{t_1}-\partial_{t_2},\\ 
\phi_1(0,1) & = & t_1\partial_{t_1}+t_2\partial_{t_2},\\
\phi_2(1) & = &
(t_1\partial_{t_1}+t_2\partial_{t_2},-(\partial_{t_1}-\partial_{t_2})).
\end{eqnarray*}
\end{example}
 
Bifiltered free resolutions have been introduced by T.\ Oaku and N.\ Takayama\cite{oaku01}. The same authors\cite{oaku01b} show that such resolutions are at the heart of the computation of important objects in $D$-module theory: the restriction of a $D$-module along a smooth subvariety, the algebraic local cohomology, the tensor product and localization.

However, the $F$-filtration they use is slightly different from ours: that is the filtration by the total order (i.e.\ $x_i$, $t_i$, $\partial_{x_i}$ and $\partial_{t_i}$ are given the weight $1$ for all $i$), which allows them to define the notion of a \emph{minimal} bifiltered free resolution.
The ranks $r_i$, called Betti numbers, and the shifts $\mathbf{n}^{(i)},\mathbf{m}^{(i)}$ do not depend on the minimal bifiltered free resolution, but do depend on the bifiltration $(F_{d,k}(M))_{d,k}$.
Our choice for the filtration $F$ comes from the common use of it, together with the $V$-filtration, in the theory of slopes and irregularity for $D$-modules, see e.g.\ the course of Y.\ Laurent\cite{laurent04}. In our setting the notion of a \emph{minimal} bifiltered free resolution no longer makes sense. However, we will derive from the Betti numbers and shifts of an arbitrary bifiltered free resolution an invariant of the $D$-module $M$: the multidegree. 

\section{Multidegree for bifiltered $D$-modules}

The multidegree has been introduced by E.\ Miller\cite{stu05} in a commutative multigraded context. It is a generalization of the notion of the degree known in projective geometry. We adapt it to the setting of bifiltered $D$-modules. First, we define the $K$-polynomial of a bifiltered $D$-module.Let 
\[
0\to D^{r_{\delta}}[\mathbf{n}^{(\delta)}][\mathbf{m}^{(\delta)}]
\stackrel{\phi_{\delta}}{\to}\cdots
\stackrel{\phi_1}{\to} D^{r_{0}}[\mathbf{n}^{(0)}][\mathbf{m}^{(0)}]
\stackrel{\phi_0}{\to} M\to 0
\]
be a bifiltered free resolution of $M$.

\begin{definition}
The $K$-polynomial of $D^r[\mathbf{n}][\mathbf{m}]$ with respect to $(F,V)$ is defined by 
\[
K_{F,V}(D^r[\mathbf{n}][\mathbf{m}];T_1,T_2)=
\sum_{i=1}^r T_1^{\mathbf{n}_i} T_2^{\mathbf{m}_i}
\in \mathbb{Z}[T_1,T_2,T_1^{-1},T_2^{-1}].
\]
The $K$-polynomial of $M$ with respect to $(F,V)$ is defined by 
\[
K_{F,V}(M;T_1,T_2)=\sum_{i=0}^{\delta} 
(-1)^i K_{F,V}(D^{r_i}[\mathbf{n}^{(i)}][\mathbf{m}^{(i)}];T_1,T_2)
 \in \mathbb{Z}[T_1,T_2,T_1^{-1},T_2^{-1}].
\]
\end{definition}

The definition of $K_{F,V}(M;T_1,T_2)$ does not depend on the bifiltered free resolution (Proposition 3.2 of our previous paper\cite{moi11}), thus that is an invariant of the bifiltered module $(M,(F_{d,k}(M))_{d,k})$ very close to the data of the Betti numbers and shifts. 

\begin{example}[Continuation of Example \ref{ex1}]\label{ex1b} 
From the bifiltered free resolution
\[
0\to D^1[2][1]\stackrel{\phi_2}{\to} D^2[1,1][1,0]\stackrel{\phi_1}{\to} D^1[0][0]\to M\to 0,
\]
we compute the $K$-polynomial $K(M;T_1,T_2)=1-(T_1T_2+T_1)+T_1^2T_2$.
\end{example}

But $K_{F,V}(M;T_1,T_2)$ depends on the bifiltration chosen, as shown in the following example.

\begin{example}\label{ex2}
Let $D=\mathbb{C}[t]\langle\partial_t\rangle$ and $M=D$ endowed with the good bifiltration $F_{d,k}(M)=F_{d,k}(D)$.
Then $M$ admits the bifiltered free resolution 
\[
0\to D[0][0]\to M\to 0
\]
 and thus $K(M;T_1,T_2)=1$. Now let 
\[
M'=\frac{D^2}{D.(1,1)}.
\]
We have an isomorphism $M\simeq M'$, given by $1\mapsto \overline{(1,0)}$. We endow $M'$ with the bifiltration 
\[
F_{d,k}(M')=\frac{F_{d,k}(D^2[1,0][0,1])+D.(1,1)}{D.(1,1)}.
\]
Then we have the following bifiltered free resolution:
\[
0\to D[1][1]\stackrel{\phi}{\to} D^2[1,0][0,1]\to M'\to 0,
\]
with $\phi(1)=(1,1)$. Then $K(M',T_1,T_2)=T_1+T_2-T_1T_2$.
\end{example} 

We now define the multidegree.

\begin{definition}\label{def1}
 $K_{F,V}(M;1-T_1,1-T_2)$ is a well-defined power series in $T_1,T_2$.
We denote by $\mathcal{C}_{F,V}(M;T_1,T_2)$ the sum of the terms whose total degree in $T_1,T_2$ equals $\textrm{codim}\,M$ in the expansion of 
$K_{F,V}(M;1-T_1,1-T_2)$. This is called the \emph{multidegree} of $M$ with respect to $(F,V)$.
\end{definition}

\begin{example}[Continuation of Example \ref{ex1b}]
We have $\mathrm{codim}\,M=2$ and 
\begin{eqnarray*}
K(M;1-T_1,1-T_2)  & = & 1-((1-T_1)(1-T_2)+(1-T_1)) + ((1-T_1)^2(1-T_2))\\
 & = & (T_1^2+T_1T_2)-T_1^2T_2,
\end{eqnarray*} 
thus $\mathcal{C}(M;T_1,T_2)=T_1^2+T_1T_2$.
\end{example}

The multidegree $\mathcal{C}_{F,V}(M;T_1,T_2)$ is a coarser invariant than $K(M;T_1,T_2)$, but its advantage is that it does not depend on the good bifiltration. In Example \ref{ex2}, we have 
$\mathcal{C}_{F,V}(M;T_1,T_2)=\mathcal{C}_{F,V}(M';T_1,T_2)=1$.

\begin{theorem}\label{thm1}
$\mathcal{C}_{F,V}(M;T_1,T_2)$ does not depend on the good bifiltration of $M$.
\end{theorem}

\begin{proof}
This theorem is similar to Theorem 3.1 of our previous paper\cite{moi11}, proved using an argument from Y.\ Laurent and T.\ Monteiro-Fernandes\cite{laurent88}. But our definition of the multidegree is slightly simpler than that given in our previous paper\cite{moi11}. Let $\mathbb{K}$ denote the fraction field of $\mathbb{C}[x]$ and $\mathcal{R}_V(M)$ denote the Rees module associated with $M$ considered as a $V$-filtered module. $\mathcal{R}_V(M)$ is naturally endowed with a $F$-filtration. In our previous paper\cite{moi11}, we have defined the multidegree as the sum of the terms whose total degree equals 
$\textrm{codim}(\mathbb{K}\otimes\textrm{gr}^{F}(\mathcal{R}_{V}(M)))$ in the expansion of 
$K_{F,V}(M;1-T_1,1-T_2)$. Here for the sake of simplicity we no longer use $\mathbb{K}$: we define the multidegree as the sum of the terms whose total degree equals $\textrm{codim}(\textrm{gr}^{F}(\mathcal{R}_{V}(M)))$. The proof of the invariance also works.

The remaining problem, so as to be in accordance with Definition \ref{def1}, is to prove that 
$\textrm{codim}(\textrm{gr}^{F}(\mathcal{R}_{V}(M)))=\textrm{codim}\,M$. We postpone the proof of it to Section 4. 
\end{proof}

\section{Examples from the theory of hypergeometric systems}

Let us consider a class of $D$-modules introduced by I.\ M.\ Gelfand, M.\ M.\ Kapranov and A.\ V.\ Zelevinsky\cite{GKZ}, generalizing the Gauss hypergeometric equations, called $A$-hypergeometric systems. They are constructed as follows. In this section $D=\mathbb{C}[x_1,\dots,x_n]\langle \partial_1,\dots,\partial_n\rangle$. Let $A$ be a $d\times n$ integer matrix and $\beta_1,\dots,\beta_d$ be complex numbers. We assume that the abelian group generated by the columns $a_1,\dots,a_n$ of $A$ is equal to $\mathbb{Z}^d$. One defines first the toric ideal $I_A$: that is the ideal of 
$\mathbb{C}[\partial_1,\dots,\partial_n]$ generated by the elements $\partial^u-\partial^v$ with $u,v\in\mathbb{N}^n$ such that $A.u=A.v$. Then one defines the hypergeometric ideal $H_A(\beta)$: that is the ideal of $D$ generated by $I_A$ and the elements 
$\sum_j a_{i,j}x_j\partial_j-\beta_i$ for $i=1,\dots,d$.  Finally the $A$-hypergeometric system $M_A(\beta)$ is defined by the quotient $D/H_A(\beta)$. 
A.\ Adolphson\cite{ado} proved (in the general case) that these $D$-modules are holonomic.  
As do M.\ Schulze and U.\ Walther\cite{SW}, we assume that the columns of $A$ lie in a single open halfspace.
The book of M.\ Saito, B.\ Sturmfels and N.\ Takayama\cite{SST} is a complete reference on (homogeneous) $A$-hypergeometric systems.

Our purpose in that section is to give some calculations of multidegree for hypergeometric systems. We will make the $V$-filtration of $D$ vary, but the module $M_A(\beta)$ will always be endowed with the bifiltration 
$F_{d,k}(M_A(\beta))=(F_{d,k}(D)+H_A(\beta))/H_A(\beta)$ once the $V$-filtration of $D$ is chosen. The computations are done using the computer algebra systems Singular\cite{singular} and Macaulay2\cite{M2}. 

\subsection{$V$-filtration along the origin}

At first we consider the $V$-filtration along $x_1=\dots=x_n=0$. 
\begin{example}\label{ex3}
 Let 
$$
A=\left(
\begin{array}{cccc}
1 & 1& 1 & 1\\
0 & 1 & 2 & 3
\end{array}
\right).
$$
Then $I_A$ is generated by 
$\partial_2\partial_4-\partial_3^2,
\partial_1\partial_4-\partial_2\partial_3,
\partial_1\partial_3-\partial_2^2.$ The ideal $H_A(\beta)$ is generated by $I_A$ and the elements $x_1\partial_1+x_2\partial_2+x_3\partial_3+x_4\partial_4-\beta_1$; $x_2\partial_2+2x_3\partial_3+3x_4\partial_4-\beta_2$.
For all $\beta$, we have
\[
\mathcal{C}_{F,V}(M_A(\beta);T_1,T_2)=3T_1^4+6T_1^3T_2+3T_1^2T_2^2
=3T_1^2(T_1+T_2)^2.
\]
\end{example}
Let  $\mathrm{vol}(A)$ denote the normalized volume 
(with $\mathrm{vol}([0,1]^d)=d!$) of the convex hull in $\mathbb{R}^d$ of the set $\{0,a_1,\dots,a_n\}$. Then we have in Example \ref{ex3}:
\begin{equation}\label{eq2}
\mathcal{C}_{F,V}(M_A(\beta);T_1,T_2)=\mathrm{vol}(A).T_1^d(T_1+T_2)^{n-d}.
\end{equation} 
Let us take a homogenizing variable $h$ and let $H(I_A)\subset\mathbb{C}[\partial,h]$ denote the homogenization of $I_A$ with respect to the total order in $\partial_1,\dots,\partial_n$. We proved the following (Theorem 5.2 of our previous paper\cite{moi11}): if 
$\mathbb{C}[\partial,h]/H(I_A)$ is a Cohen-Macaulay ring and the parameters $\beta_1,\dots,\beta_d$ are generic, then the formula (\ref{eq2}) holds.

\begin{question}
Does the formula (\ref{eq2}) hold if $\beta_1,\dots,\beta_d$ are generic, but without the Cohen-Macaulay assumption ? 
\end{question}
For a holonomic module $M$, let us write $\mathcal{C}_{F,V}(M;T_1,T_2)=
b_0(M)T_1^n+b_1(M)T_1^{n-1}T_2+\dots+b_n(M)T_2^n$. From (\ref{eq2}) we have $b_0(M_A(\beta))=\mathrm{vol}(A)$.
We claim that the latter equality holds for all $A$, if $\beta_1,\dots,\beta_d$ are generic. Indeed if $M$ is any holonomic $D$-module we have $b_0(M)=\mathrm{rank}(M)$ (Remark 5.1 of our previous paper\cite{moi11}), where $\mathrm{rank}(M)$, the holonomic rank of $M$, is the dimension of the vector space of local holomorphic solutions of $M$ (considered as a system of linear partial differential equations) at a generic point. Let us remark that the niceness assumption in Remark 5.1 of our previous paper\cite{moi11} can be dropped because of Proposition \ref{prop1}. Now, by Adolphson\cite{ado},
 $\mathrm{rank}(M_A(\beta))=\mathrm{vol}(A)$ for generic $\beta_1,\dots,\beta_d$, which concludes to prove our claim. 
 
\begin{example}[from Saito-Sturmfels-Takayama\cite{SST}]\label{ex4}
 Let 
$$
A=\left(
\begin{array}{cccc}
1 & 1& 1 & 1\\
0 & 1 & 3 & 4
\end{array}
\right).
$$
Then $I_A$ is generated by 
$\partial_2\partial_4^2-\partial_3^3,\partial_1\partial_4-\partial_2\partial_3,\partial_1\partial_3^2-\partial_2^2\partial_4,\partial_1^2\partial_3-\partial_2^3$. 
Here $\mathbb{C}[\partial,h]/H(I_A)$ is not Cohen-Macaulay. However for $(\beta_1,\beta_2)\neq (1,2)$, we have 
\[
\mathcal{C}_{F,V}(M_A(\beta);T_1,T_2)=4T_1^4+8T_1^3T_2+4T_1^2T_2^2
=4T_1^2(T_1+T_2)^2,
\]
which agrees with the formula (\ref{eq2}).

For $(\beta_1,\beta_2)= (1,2)$, we have 
\[
\mathcal{C}_{F,V}(M_A(\beta);T_1,T_2)=
5T_1^4+12T_1^3T_2+10T_1^2T_2^2+4T_1T_2^3+T_2^4.
\]
Let us remark that $(T_1+T_2)^2$ is still a factor, indeed 
$$
5T_1^4+12T_1^3T_2+10T_1^2T_2^2+4T_1T_2^3+T_2^4=
(T_1+T_2)^2(5T_1^2+2T_1T_2+T_2^2).
$$
\end{example}  

\subsection{$V$-filtration along coordinate hyperplanes}

Let us reconsider Examples \ref{ex3} and \ref{ex4}, with the $V$-filtration along $x_i=0$ for some fixed $i$. 

\begin{example}[continuation of Example \ref{ex3}] $ $

\begin{itemize}
\item $V$-filtration along $x_1=0$: for any $\beta_1,\beta_2$,
$$\mathcal{C}_{F,V}(M_A(\beta))=3T_1^4+2T_1^3T_2.$$
\item $V$-filtration along $x_2=0$: for any $\beta_1,\beta_2$,
$$\mathcal{C}_{F,V}(M_A(\beta))=3T_1^4+3T_1^3T_2.$$
\item $V$-filtration along $x_3=0$: same as $V$-filtration along $x_2=0$.
\item $V$-filtration along $x_4=0$: same as $V$-filtration along $x_1=0$.
\end{itemize}
\end{example}  

\begin{example}[Continuation of Example \ref{ex4}]\label{ex4b}
Let us take the $V$-filtration along $x_4=0$.
Then if $(\beta_1,\beta_2)\neq(1,2)$, then
$$\mathcal{C}_{F,V}(M_A(\beta))=4T_1^4+3T_1^3T_2.$$
For  $(\beta_1,\beta_2)=(1,2)$, we have
$$\mathcal{C}_{F,V}(M_A(\beta))=5T_1^4+4T_1^3T_2.$$
\end{example}

\subsection{Dependency of the multidegree on the parameters}

Studying the dependency of the multidegree on the parameters $\beta_1,\dots,\beta_d$ is a natural problem. A basic known fact is that the multidegree remains constant outside of an algebraic hypersurface of $\mathbb{C}^d$ (analogous to Proposition 1.5 of our previous paper\cite{moi11}). 
The depedency of the holonomic rank of $M_A(\beta)$ has been deeply studied by several authors, for instance in Saito-Sturmfels-Takayama\cite{SST} and L.F.\ Matusevich, E.\ Miller and U.\ Walther\cite{MMW}. In particular it has been proved in the latter paper that the holonomic rank of $M_A(\beta)$ is upper semi-continuous as a function on the parameter $(\beta_1,\dots,\beta_d)$, which means that the holonomic rank at an exceptionnal parameter 
is greater than the holonomic rank at a generic parameter. Thus the coefficient $b_0(M_A(\beta))$ is an upper semi-continuous function on $\beta_1,\dots,\beta_d$. Nobuki Takayama pointed out the problem of that semi-continuity for the other coefficients of the multidegree, as also suggest Examples \ref{ex4} and \ref{ex4b}. 

\begin{question}
Let us fix the matrix $A$, the $V$-filtration on $D$ and $0\leq i\leq n$.
Is the coefficient $b_i(M_A(\beta))$ a upper semi-continuous function on $\beta_1,\dots,\beta_d$ ?
\end{question}

\subsection{Positivity}
The following positivity problem is due to an observation by Michel Granger.

\begin{question}
Let us fix $A$, $\beta_1,\dots,\beta_d$, the $V$-filtration on $D$ and $0\leq i\leq n$.
Do we always have $b_i(M_A(\beta))\geq 0$ ?
\end{question} 
That is the case in all the examples we considered. Moreoever, since $b_0(M_A(\beta))=\mathrm{rank}(M_A(\beta))$, the positivity is true for $b_0(M_A(\beta))$.

\section{Proof of Theorem \ref{thm1}}

Here we complete the proof of Theorem \ref{thm1}. 
Again $D=\mathbb{C}[x,t]\langle\partial_x,\partial_t\rangle.$
We recall the definition of the Rees ring $\mathcal{R}_V(D)=\oplus V_k(D)T^k$ endowed with the filtration $F_d(\mathcal{R}_V(D))=\oplus_k F_{d,k}(D)T^k$. 
We have a ring isomorphism $\mathcal{R}_V(D)\simeq D[\theta]$ given by
$x_iT^0\mapsto x_i$; $t_iT^{-1}\mapsto t_i$; $T^1\mapsto\theta$; $\partial_{x_i}T^0\mapsto\partial_{x_i}$ and $\partial_{t_i}T^1\mapsto\partial_{t_i}$. The $F$-filtration on $D[\theta]$ induced by this isomorphism is given by assigning the weights $(\mathbf{0},\mathbf{0},\mathbf{1},\mathbf{1},0)$ to $(\mathbf{x},\mathbf{t},\partial_{\mathbf{x}},\partial_{\mathbf{t}},\theta)$. Suppose that, as a bifiltered $D$-module, $M=D^r[\mathbf{n}][\mathbf{m}]/N$. We endow $M$ with the $V$-filtration $V_k(M)=V_k(D^r[\mathbf{m}]/N)$. 

Then we associate with $M$ a Rees module $\mathcal{R}_V(M)=\oplus V_k(M)T^k$ endowed with the $F$-filtration $F_d(\mathcal{R}_V(M))=\oplus_k F_{d,k}(M)T^k$.
We have an isomorphism of graded modules  
\[
\frac{D[\theta]^r[\mathbf{m}]}{H^V(N)}\simeq \mathcal{R}_V(M)
\] 
given by $\overline{e}_i\mapsto\overline{e}_iT^{m_i}$, with $(e_i)$ the canonical base either of $D[\theta]^r$ or of $D^r$, and $H^V(N)$ the homogenization of $N$ with respect to $V$. Furthermore it is an isomorphism of $F$-filtered modules
 \[
\frac{D[\theta]^r[\mathbf{n}]}{H^V(N)}\simeq \mathcal{R}_V(M).
\] 

We denote by $\mathrm{codim}\mathcal{R}_V(M)$ the codimension of the module $\mathrm{gr}^F(\mathcal{R}_V(M))$. In fact it is allowed to replace the weight vector 
$(\mathbf{0},\mathbf{0},\mathbf{1},\mathbf{1},0)$ defining the $F$-filtration by any non-negative weight vector $G$, giving rise to a filtration on $D[\theta]$ also denoted by $G$, such that $\mathrm{gr}^G(D[\theta])$ is a commutative ring (see Proposition 5.1 of our previous paper\cite{moi11}).

\begin{proposition}\label{prop1}
$\mathrm{codim}\mathcal{R}_V(M)=\mathrm{codim}M$.
\end{proposition}

\begin{lemma}
$\mathrm{codim}\mathcal{R}_V(M)\leq\mathrm{codim}M$.
\end{lemma}

\begin{proof}
We make use of the characterization of the codimension by means of extension groups (see R.\ Hotta, K.\ Takeuchi and T.\ Tanisaki\cite{tanisaki}, Theorem D.4.3):
\begin{eqnarray*}
\mathrm{codim}M & = & 
\mathrm{inf}\{i: \mathrm{Ext}^i_{D}(M,D)\neq 0\},\\
\mathrm{codim}\mathcal{R}_V(M) & = & 
\mathrm{inf}\{i: \mathrm{Ext}^i_{\mathcal{R}_V(D)}(\mathcal{R}_V(M),\mathcal{R}_V(D))\neq 0\}.
\end{eqnarray*}
Let 
\[
\cdots\to\mathcal{L}_1\stackrel{\phi_1}{\to}\mathcal{L}_0\stackrel{\phi_0}{\to} M\to 0
\]
be a $V$-adapted free resolution of $M$, with 
$\mathcal{L}_i=D^{r_i}[\mathbf{m}^{(i)}]$.
It induces a graded free resolution
\[
\cdots\to\mathcal{R}_V(\mathcal{L}_1)\stackrel{\mathcal{R}_V(\phi_1)}{\to}
\mathcal{R}_V(\mathcal{L}_0)\stackrel{\mathcal{R}_V(\phi_0)}{\to} 
\mathcal{R}_V(M)\to 0,
\]
with $\mathcal{R}_V(\mathcal{L}_i)\simeq D[\theta]^{r_i}[\mathbf{m}^{(i)}]$.

If $N$ is a left $D$-module, let $N^*=\mathrm{Hom}_D(N,D)$. If $N$ is a left $\mathcal{R}_V(D)$-module, let 
$N^*=\mathrm{Hom}_{\mathcal{R}_V(D)}(N,\mathcal{R}_V(D))$.

The complex $\mathcal{L}^*:$
\[
0\to\mathcal{L}_0^*\stackrel{\phi_1^*}{\to}
\mathcal{L}_1^*\stackrel{\phi_2^*}{\to}\cdots 
\]
gives the groups $\mathrm{Ext}^i_{D}(M,D)$.
The complex $\mathcal{R}_V(\mathcal{L})^*:$
\[
0\to\mathcal{R}_V(\mathcal{L}_0)^*\stackrel{\mathcal{R}_V(\phi_1)^*}{\to}
\mathcal{R}_V(\mathcal{L}_1)^*\stackrel{\mathcal{R}_V(\phi_2)^*}{\to}\cdots
\]
gives the groups 
$\mathrm{Ext}^i_{\mathcal{R}_V(D)}(\mathcal{R}_V M,\mathcal{R}_VD)$.

If $N$ is a left $D$-module endowed with a good $V$-filtration, then we endow $N^*$ by the exhaustive filtration 
\[
V_k(N^*)=
\{u:N\to D\ \textrm{such that}\ \forall j, u(V_j(N))\subset V_{j+k}(D)\}.
\]
Let us consider $\phi_i:\mathcal{L}_i\to\mathcal{L}_{i-1}$ together with 
$\phi_i^*:\mathcal{L}_{i-1}^*\to\mathcal{L}_{i}^*$.
We endow $\mathrm{ker}(\phi_i^*)$ with the induced $V$-filtration.
We claim that 
\[
\mathcal{R}_V(\mathrm{ker}(\phi_i^*))\simeq\mathrm{ker}(\mathcal{R}_V(\phi_i)^*).
\]
Indeed, let $\mathcal{L}=D^r[\mathbf{m}]$ be a $V$-filtered free module.
We have a bijection $\mathcal{R}_V(\mathcal{L}^*)\simeq (\mathcal{R}_V\mathcal{L})^*$, by mapping $\oplus u_kT^k$, with $u_k\in V_k(\mathcal{L}^*)$, to $\sum\mathcal{R}_V(u_k)\in (\mathcal{R}_V\mathcal{L})^*$.
Under that bijection, $\mathcal{R}_V(\mathrm{ker}(\phi_i^*))\subset \mathrm{R}_V(\mathcal{L}_{i-1}^*)$ is seen as a subset of $(\mathcal{R}_V\mathcal{L}_{i-1})^*$ and is equal to $\mathrm{ker}(\mathcal{R}_V(\phi_i)^*)$.

On the other hand, let us endow $\mathrm{Im}(\phi_i^*)$ with the induced $V$-filtration. Using the identification $\mathcal{R}_V(\mathcal{L}_i^*)\simeq (\mathcal{R}_V\mathcal{L}_i)^*$, we have
\[
\mathrm{Im}(\mathcal{R}_V(\phi_i)^*)\subset
 \mathcal{R}_V(\mathrm{Im}(\phi_i^*)).
\]

Then if $H_i(\mathcal{L}^*)=\mathrm{ker}(\phi_{i+1}^*)/\mathrm{Im}(\phi_i^*)$ is endowed with the quotient $V$-filtration, we have
\[
H_i(\mathcal{R}_V(\mathcal{L})^*)=
\frac{\mathrm{ker}(\mathcal{R}_V(\phi_{i+1})^*)}
{\mathrm{Im}(\mathcal{R}_V(\phi_i)^*)}
\twoheadrightarrow \frac{\mathcal{R}_V(\mathrm{ker}(\phi_{i+1}^*))}
{\mathcal{R}_V(\mathrm{Im}(\phi_i^*))}
=\mathcal{R}_V(H_i(\mathcal{L}^*)).
\]
Thus if $H_i(\mathcal{R}_V(\mathcal{L})^*)=0$, then $\mathcal{R}_V(H_i(\mathcal{L}^*))=0$ which implies that 
$H_i(\mathcal{L}^*)=0$. The lemma follows.
\end{proof}

\begin{lemma}
$\mathrm{codim}M\leq\mathrm{codim}\mathcal{R}_V(M)$.
\end{lemma}

\begin{proof}
Here we make use of the theory of Gelfand-Kirillov dimension, see e.g.\ 
G.G.\ Smith\cite{smith}. 
Let now $F$ denote the Bernstein filtration on $D$, i.e. each variable $x_i,\partial_{x_i},t_i,\partial_{t_i}$ has weight $1$. We endow $M$ with the good $(F,V)$-bifiltration, still denoted by $F_{d,k}(M)$, given by the quotient $D^r[\mathbf{0}][\mathbf{m}]/N$.  

Let $\phi:\mathbb{N}\to \mathbb{R}$. Let 
$\gamma(\phi)=\mathrm{inf}\{i: f(d)\leq d^i\ \textrm{for $d$ large enough}\}$. By Gelfand-Kirillov theory, we have 
$\gamma(d\mapsto \textrm{dim}_{\mathbb{C}}F_d(M))=\textrm{dim}M=2n-\textrm{codim}M$.

 Let us define the filtration $(G_d(D[\theta]))$ by giving the weight $1$ to all the variables. Then $\textrm{gr}^G(D[\theta])$ is commutative and for any $d$, $G_d(D[\theta])$ is finitely dimensional over $\mathbb{C}$.
Then we endow $\mathcal{R}_V(M)$ with the $G$-filtration $G_d(\mathcal{R}_V(M))=G_d(D[\theta]^r[\mathbf{0}]/H^V(N))$.

Let $d\in\mathbb{N}$ and $E_d$ be the interval $[\textrm{min}_im_i-d,\textrm{max}_im_i+d]$. We have
\[
G_d(\mathcal{R}_VM)\subset \bigoplus_{k\in E_d}F_{d,k}(M)T^k
\subset \bigoplus_{k\in E_d}F_{d}(M)T^k.
\]
Then $\textrm{dim}_{\mathbb{C}}G_d(\mathcal{R}_V(M))
\leq (2d+c)\textrm{dim}_{\mathbb{C}}F_d(M)$, with 
$c=\textrm{max}_im_i-\textrm{min}_im_i+1$.
Thus 
\begin{eqnarray*}
\textrm{dim}\mathcal{R}_V(M) & = &\gamma(d\mapsto \textrm{dim}_{\mathbb{C}}G_d(\mathcal{R}_V(M)))\\
 & \leq & \gamma(d\mapsto (2d+c)\textrm{dim}_{\mathbb{C}}F_d(M))\\
 & = & 1+\gamma(d\mapsto\textrm{dim}_{\mathbb{C}}F_d(M))\\
 & = & 1+\textrm{dim}M.
\end{eqnarray*}
\end{proof}

\section*{Acknowledgements}

I sincerely thank M.\ Granger and T.\ Oaku for their comments and the Japan Society for the Promotion of Science for the financial support.

\end{document}